\documentclass[preprint,12pt]{elsarticle}
\usepackage{hyperref}
\usepackage{amsmath,amsthm}
\usepackage{amssymb}

\newcommand{\eps}{\varepsilon}
\newcommand{\Rd}{\mathbb R^d}

\newcommand{\Rp}{\mathbb R^+}
\newcommand{\E}{\mathbb E}
\newcommand{\pr}{\mathbb P}
\newcommand{\N}{\mathbb N}
\newcommand{\1}{\mathbf 1}
\newcommand{\del}{\partial}

\def \levy {L\'{e}vy}

\newtheorem{theorem}[equation]{Theorem}
\newtheorem{lemma}[equation]{Lemma}
\newtheorem{proposition}[equation]{Proposition}

\newdefinition{definition}[equation]{Definition}

\newdefinition{remark}[equation]{Remark}
\newproof{pot}{Proof of Theorem}

\numberwithin{equation}{section}

\title{Lagging/Leading Coupled Continuous Time Random Walks, Renewal Times and their Joint Limits}

\begin{document}

 \author[UNSW]{B.I. Henry}
 \ead{B.Henry@unsw.edu.au}
 \author[UNSW]{P. Straka\corref{cor}}
 \ead{p.straka@unsw.edu.au}

\cortext[cor]{Corresponding author}

 \address[UNSW]{Department of Applied Mathematics, School of Mathematics and Statistics,
 University of New South Wales, Sydney NSW 2052, Australia.}

\begin{abstract}
Subordinating a random walk to a renewal process yields a continuous time random walk (CTRW) model for diffusion, including the possibility
of anomalous diffusion.
Transition densities of scaling limits of power law CTRWs have been shown to solve fractional Fokker-Planck equations.
We consider limits of sequences of CTRWs which arise when both waiting times and jumps are taken from an infinitesimal triangular array. We identify two different limit processes $X_t$ and $Y_t$ when waiting times precede or follow jumps, respectively.
In the limiting procedure, we keep track of the renewal times of the CTRWs and hence find two more limit processes.
Finally, we calculate the joint law of all four limit processes evaluated at a fixed time $t$.
%

\end{abstract}

\maketitle

\section{Introduction}

An i.i.d.\ sequence of jumps $J_i$ in $\Rd$ separated by an i.i.d.\ sequence of positive waiting times $W_i$ yields a jump process
known as a continuous time random walk (CTRW) \cite{Montroll1965}.
Continuous time random walks, and evolution equations for their limiting distributions, obtained 
as the step size tends to zero and the number of steps tends to infinity,
have been widely studied over the past few decades as physical models
of diffusion. 
CTRWs with power law waiting time densities and/or infinite variance jumps,
and evolution equations for their limiting distributions, formulated in terms of fractional order
partial differential equations, have been of particular interest, as physical 
models for anomalous diffusion \cite{MK2000,HLS2008}.
%
%
 The limiting distributions of CTRWs  have also been investigated using a mathematical approach based on renewal theory and limit theorems for
sums of random jumps \cite{Kotulski1995,BeMS04Ltca}.
More general results have been obtained using 
 ``triangular array limits'' \cite{BKKK2007,MeSc07Tal,SiTe04Ltm}.
A statement of the problem in this context is: For every $n\in\N$,
consider a CTRW $X^n$ arising from an iid sequence
$(J^n_i,W^n_i)_{i\in\N}$ with common law $\Pi^n$ on $\Rd\times\Rp$. Then assume that $\Pi^n$ converges weakly to the Dirac measure concentrated at $(0,0)$, and consider possible stochastic process limits of $X^n$.

In this paper, we consider coupled CTRWs where the $J_i$ are not independent of $W_i$. 
These are of particular interest in finance \citep{scalas2000fractional,MeSc06Cct}, as there is empirical evidence for stock markets with correlated waiting times and log-returns \citep{raberto2002waiting}; also see \cite{BeMS04Ltca} for a comprehensive list of coupled CTRWs having appeared in the literature.
We show that in general there are two different limit processes $X$ and $Y$ that arise when each $W_i$ precedes or succeeds $J_i$, respectively.
In the transition to the limit, 
we also keep track of the processes $G^n$ and $D^n$ given by the last renewal times $G^n_t$ before $t$ and first renewal times $D^n_t$ after $t$, and show that they jointly converge with $X$ and $Y$.
As $X$ and $Y$ turn out to be constant on every interval $[G_t,D_t)$, this enables us to 
model the time intervals in which the diffusing particle is trapped.
After defining the age process $A$ and the remaining lifetime process $R$ via  $A_t = t - G_t$ and $R_t = D_t - t$, respectively,
 we calculate the joint laws of $(X_t,A_t,Y_t,R_t)$ for each fixed time $t\geq 0$.
%
Finally, we study the ageing behaviour of $X$ and $Y$.

The remainder of this paper is organized as follows:
In section 2 we introduce our notation and establish general properties for functions that
are right-continuous with left-hand limits (\textbf{rcll}) and left continuous with right-hand limits (\textbf{lcrl}) and define two continuous mappings on a measurable subset of Skorohod space.
In section 3 we consider triangular array limits for CTRWs and we define the stochastic processes ``lagging CTRW'' $X^n$, ``leading CTRW'' $Y^n$, ``last time of renewal'' $G^n$ and ``next time of renewal'' $D^n$.
Using the continuous mapping theorem, we prove limit theorems for the weak convergence of these processes (Theorem \ref{th:CTRW-limit}).
Finally, in section 4 we consider the stochastic processes $A$ and $R$ corresponding to the age and remaining lifetime
respectively and we obtain an integral equations for the joint law of $(X_t,A_t,Y_t,R_t)$ (Theorem \ref{th:munu}).

\section{Two continuous mappings on Skorohod space}

For a separable complete metric space $E$, let $\mathbb D(E)$ be the set of all functions defined on $\Rp := [0,\infty)$ with values in $E$ which are right-continuous and have limits from the left (in short, the set of all \textbf{rcll paths}). We assume that $\mathbb D(E)$ is endowed with the (metrizable) Skorohod topology $J$ \citep{Ja87}.
Equipped with the corresponding Borel-$\sigma$-algebra $\mathcal D(E)$, $\mathbb D(E)$ is a measurable space.

\paragraph{Skorohod Subspaces}
For an element $\alpha \in \mathbb D(\Rd\times\Rp)$, we write $\alpha = (\beta, \sigma)$, where $\beta \in \mathbb D(\Rd)$ and $\sigma \in \mathbb D(\Rp)$.
We write $D_u$, $D_\uparrow$ and $D_\upuparrows$ for the sets of all such $\alpha$ which have unbounded,  non-decreasing and increasing $\sigma$, respectively.
Slight variations of the proofs of \cite[lem.13.2.3, lem.13.6.1]{Whi02} show that $D_\uparrow$ is closed in $\mathbb D(\Rd\times\Rp)$, $D_u$ is a $G_\delta$-subset of $\mathbb D(\Rd\times\Rp)$ (i.e.\ a countable intersection of open subsets) and $D_\upuparrows$ is a $G_\delta$-subset of $D_\uparrow$.
Hence we have:
\begin{lemma}\label{lem:measurable}
The sets $D_{\uparrow,u} := D_\uparrow \cap D_u$ and $D_{\upuparrows,u} := D_\upuparrows \cap D_u$ are Borel measurable.
\end{lemma}

For an rcll path $\xi$ we write $\xi^-$ for its \textbf{lcrl} version (the corresponding left-continuous path having right-hand limits) given by
$$t \mapsto \lim_{\eps \downarrow 0}\xi(t-\eps),~~ t>0, ~~~~~~ \xi^-(0) = \xi(0),$$
and similarly, if $\xi$ is lcrl, we write $\xi^+$ for its rcll version
$$t\mapsto \lim_{\eps\downarrow 0}\xi(t+\eps), ~~~~~~ t\geq 0.$$
It will be convenient to use both notations $\xi^-(t) = \xi(t-)$ and $\xi^+(t) = \xi(t+)$.
For an unbounded $\xi\in\mathbb D(\Rp)$ we define its generalized inverse via
$$\xi^{-1}(t):= \inf\{s\geq 0: \sigma(s)>t\}$$
and note that $\xi^{-1}$ is a non-decreasing, unbounded element of $\mathbb D(\Rp)$.
\begin{definition}
\label{def:Phi}
For $\alpha = (\beta,\sigma)\in D_{\uparrow,u}$ write $\ell_\alpha := \sigma^{-1}$. Then put
\begin{align*}
 \Phi: D_{\uparrow,u} &\to \mathbb D(\Rd\times\Rp)  &\text{ and }&& \Psi: D_{\uparrow,u}&\to \mathbb D(\Rd\times\Rp) \\
\alpha  &\mapsto \left(\alpha^-\circ\ell_\alpha^-\right)^+ &&& \alpha &\mapsto \alpha\circ\ell_\alpha
\end{align*}
\end{definition}

Since $\ell_\alpha$ is non-decreasing and rcll, $\alpha\circ\ell_\alpha$ is rcll, and so $\Psi$ is well-defined. Moreover it is not hard to see that $\alpha^-\circ\ell_\alpha^-$ is lcrl and hence that $\Phi$ is well-defined.

Note that the topology $J$ on $\mathbb D(\Rd\times\Rp)$ induces the \textit{relative topology} or \textit{subspace topology} on the subset $D_{\uparrow,u}$. The following is the key ingredient of the continuous mapping theorem in section \ref{sec:CTRW-limit}:
\begin{proposition}
\label{prop:PhiCont}
  The mappings $\Phi$ and $\Psi$ are continuous at $D_{\upuparrows,u}$.
\end{proposition}

The proof is based on the characterization of convergence in $\mathbb D(E)$ taken from \cite[th~3.6.5]{EtK86}; for convenience, it is restated in proposition \ref{prop:EK}.
We use the abbreviation $x_n\to A$ if $A\subset \mathbb D(E)$ contains all limit points of the sequence $\{x_n\}$. 
\begin{proposition}
\label{prop:EK}
Let $\{x_n\}_{n\in\N}\subset \mathbb D(E)$ and $x\in \mathbb D(E)$.
Then $x_n\to x$ in $\left(\mathbb D(E),J\right)$ if and only if
whenever $\{t_n\}_{n\in\N}\subset \Rp$, $t\geq 0$, and $\lim t_n = t$, the following conditions hold:
\begin{enumerate}[(I)]
\item
$x_n(t_n) \to \{x(t),x^-(t)\}$
\item
If $x_n(t_n)\to x(t)$ and $\{s_n\}_{n\in\mathbb N} \subset \Rp$
is such that $s_n\geq t_n$, $s_n\to t$, then
$x_n(s_n)\to x(t)$.
\item
If $x_n(t_n)\to x^-(t)$ and $\{s_n\}_{n\in\mathbb N} \subset \Rp$
is such that $0\leq s_n\leq t_n$, $s_n\to t$, then
$x_n(s_n)\to x^-(t) $.
\end{enumerate}
\end{proposition}
We will also need the corresponding version for lcrl paths:
\begin{lemma}
 Proposition \ref{prop:EK} holds if conditions (I), (II) and (III) are replaced by
\begin{enumerate}[(I$^-$)]
\item
$x^-_n(t_n) \to \{x(t),x^-(t)\}$
\item
If $x^-_n(t_n)\to x(t)$ and $\{s_n\}_{n\in\mathbb N} \subset \Rp$
is such that $s_n\geq t_n$, $s_n\to t$, then
$x^-_n(s_n)\to x(t)$.
\item
If $x^-_n(t_n)\to x^-(t)$ and $\{s_n\}_{n\in\mathbb N} \subset \Rp$
is such that $0\leq s_n\leq t_n$, $s_n\to t$, then
$x^-_n(s_n)\to x^-(t) $.
\end{enumerate}
\begin{proof}
Without loss of generality we assume $t_n > 0$ for all $n$. Then there is a sequence $\{\eps_n\}_{n\in\N}$, $t_n> \eps_n >0$, $\lim \eps_n = 0$, such that
\begin{align*}
  d(x_n^-(t_n),x_n(t_n-\eps_n)) & \to 0, & d(x_n(t_n),x_n^-(t_n+\eps_n)) &\to0.
\end{align*}
Since $\lim t_n-\eps_n = \lim t_n + \eps_n = t$, the equivalence (I)$\Leftrightarrow$(I$^-$) follows. Suppose now that (II) holds, and that $s_n\geq t_n$ is such that $\lim s_n = \lim t_n = t$, $x^-_n(t_n) \to x(t)$. Then there are sequences $\eps_n$ and $\eps'_n$ tending to $0$ and satisfying
\begin{align*}
  d(x^-(t_n),x(t_n-\eps_n)) & \to 0, & d(x^-(s_n),x(s_n-\eps'_n)) \to 0
\end{align*}
which we can choose in such a way that $0<\eps'_n<\eps_n$ for all $n$. Then $s_n - \eps'_n \geq t_n - \eps_n$ and $s_n - \eps'_n\to t$, hence
\begin{align*}
  x(t) = \lim x^-_n(t_n) = \lim x(t_n-\eps_n) = \lim x(s_n - \eps'_n) = \lim x^-(s_n),
\end{align*}
where the third equality follows from (II). This shows the implication (II)$\Rightarrow$(II$^-$). The remaining parts (II$^-$)$\Rightarrow$(II) and (III)$\Leftrightarrow$(III$^-$) are shown similarly.
\end{proof}
\end{lemma}

\begin{proof}[Proof of proposition~\ref{prop:PhiCont}]
Throughout this proof, let $\{\alpha_n\}_{n\in\N}\subset D_{\uparrow,u}$, $\alpha\in D_{\upuparrows,u}$,
$\alpha_n\to\alpha$ in $\mathbb D(\Rd\times\Rp)$ with respect to the $J$-topology, $t\geq 0$,  $\{t_n\}_{n\in\N}\subset \Rp$,
$\{s_n\}_{n\in\N}\subset \Rp$, $t_n \to t$, $s_n \to t$.
We put $\gamma_n:= \Phi(\alpha_n)$, $\gamma := \Phi(\alpha)$, $\delta_n := \Psi(\alpha_n)$, $\delta := \Psi(\alpha)$
and show that (I), (II) and (III) are satisfied with $x_n$ and $x$ replaced by $\delta_n$ and $\delta$, and that (I$^-$), (II$^-$) and (III$^-$) are satisfied with $x_n$ and $x$ replaced by $\gamma_n$ and $\gamma$.

First, note that by \cite[cor.13.6.4]{Whi02},
\begin{align}
\label{eq:tauconv}
\ell_{\alpha_n} \to \ell_\alpha ~~~ \text{ in } \mathbb D(\Rp).
\end{align}
We define $\tau_n:=\ell_{\alpha_n}(t_n)$, $\tau := \ell_{\alpha}(t)$, $\xi_n := \ell_{\alpha_n}(s_n)$.
Since $\alpha \in D_{\upuparrows,u}$, $\ell_{\alpha}$ is continuous, and hence (I) applied to $\ell_{\alpha_n}$ yields $\lim \tau_n = \lim \xi_n =  \tau$.

If $t=0$, then $\tau = 0$ and $\delta_n(t_n) = \alpha_n(\tau_n) \to \alpha(0) = \delta(0)$ by (I) applied to $\alpha_n$, and one checks that (I), (II) and (III) are satisfied for $\delta_n$.

Assume now that $t>0$, and that $\ell_\alpha$ is left-increasing at $t$, i.e.\ $s<t \Rightarrow \ell_\alpha(s)<\ell_\alpha(t)$. Then $\delta^-(t) = \alpha^-(\tau)$, and
\begin{align*}
  \delta_n(t_n) = \alpha_n(\tau_n) \to \{\alpha^-(\tau), \alpha(\tau)\} = \{\delta(t), \delta^-(t)\},
\end{align*}
showing (I) for $\delta_n$. In order to show (II) (resp.\ (III)) for $\delta_n$, suppose $\delta_n(t_n) \to \delta(t)$ (resp.\ $\delta_n(t_n) \to \delta^-(t)$). Then
\begin{align*}
\alpha_n(\tau_n) = \delta_n(t_n) &\to \delta(t) = \alpha(\tau)  &({\rm resp.}&&  \alpha_n(\tau_n) = \delta_n(t_n) &\to \delta^-(t) = \alpha^-(\tau)).
\end{align*}
Since $\ell_{\alpha_n}$ is non-decreasing, $s_n\geq t_n$ (resp.\ $s_n\leq t_n$) implies $\xi_n \geq \tau_n$ (resp.\ $\xi_n \leq \tau_n$).
Then (II) (resp.\ (III)) applied to $\alpha_n$ yields
\begin{align*}
\lim \delta_n(s_n) = \lim \alpha_n(\xi_n) = \alpha(\tau)  &= \delta (t),
\\ ({\rm resp.}~~ \lim \delta_n(s_n) = \lim \alpha_n(\xi_n) = \alpha^-(\tau)  &= \delta^-(t),~~)
\end{align*}
showing (II) (resp.\ (III)) for $\delta_n$.

If $\ell_\alpha$ is left-constant at $t$, i.e.\ not left-increasing, then $\alpha(\ell_\alpha(t-))=\alpha(\ell_\alpha(t))$ and hence $\delta$ is continuous at $t$. Hence (I),(II) and (III) for $\delta_n$ reduces to showing $\delta_n(t_n) \to \delta(t)$.
(I) applied to $\alpha_n$ yields $\delta_n(t_n) = \alpha_n(\tau_n) \to \{\alpha(\tau), \alpha^-(\tau)\}$.
Since $\ell_\alpha$ is left-constant at $t$, we have $\sigma^-(\tau)<t\leq \sigma(\tau)$. As $\sigma$ is the last coordinate of $\alpha$, we see that $\alpha(\tau) \neq \alpha^-(\tau)$. Suppose now that $\alpha_n(\tau_n)\to \alpha^-(\tau)$. It follows that $\sigma_n(\tau_n)\to \sigma^-(\tau)<t$, which contradicts $\sigma_n(\tau_n)\geq t_n \to t$.
Thus $\alpha_n(\tau_n)\to \alpha(\tau)$ and $\delta_n(t_n) \to \delta(t)$.
We have proven the continuity statement about $\Psi$.

We turn to $\gamma_n$ and the conditions (I$^-$), (II$^-$) and (III$^-$), and define
$\tau^-_n:=\ell^-_{\alpha_n}(t_n)$ and $\xi^-_n:= \ell^-_{\alpha_n}(s_n)$.
As before, (\ref{eq:tauconv}), the continuity of $\ell_\alpha$ and (I) applied to $\ell_{\alpha_n}$ yield
$\lim \tau^-_n = \lim \xi^-_n = \tau$.
Note that $\gamma^- = \alpha^-\circ \ell_\alpha$.  Assume first that $\ell_\alpha$ is right-increasing at $t$. Then $\gamma(t) = \alpha(\tau)$, and
\begin{align*}
  \gamma^-_n(t_n) = \alpha^-_n(\tau^-_n) \to \{\alpha(\tau), \alpha^-(\tau)\} = \{\gamma(t), \gamma^-(t)\},
\end{align*}
by (I$^-$) applied to $\alpha^-_n$, showing (I$^-$) for $\gamma^-_n$.
In order to show (II$^-$) (resp.\ (III$^-$)) for $\gamma^-_n$, suppose $\gamma^-_n(t_n) \to \gamma(t)$ (resp.\ $\gamma^-_n(t_n) \to \gamma^-(t)$). Then
\begin{align*}
\alpha^-_n(\tau^-_n) = \gamma^-_n(t_n) &\to \gamma(t) = \alpha(\tau)  &({\rm resp.}&&  \alpha^-_n(\tau^-_n) = \gamma^-_n(t_n) &\to \gamma^-(t) = \alpha^-(\tau)).
\end{align*}
Since $\ell_{\alpha_n}$ and $\ell^-_{\alpha_n}$ are non-decreasing, $s_n\geq t_n$ (resp.\ $s_n\leq t_n$) implies $\xi^-_n \geq \tau^-_n$ (resp.\ $\xi^-_n \leq \tau^-_n$).
Then (II$^-$) (resp.\ (III$^-$)) applied to $\alpha_n$ yields
\begin{align*}
\lim \gamma^-_n(s_n) = \lim \alpha_n(\xi^-_n) = \alpha(\tau)  &= \gamma (t),
\\ ({\rm resp.}~~ \lim \gamma^-_n(s_n) = \lim \alpha_n(\xi^-_n) = \alpha^-(\tau)  &= \gamma^-(t),~~)
\end{align*}
showing (II$^-$) (resp.\ (III$^-$)) for $\gamma^-_n$.

If $\ell_\alpha$ is right-constant at $t$, i.e.\ not right-increasing, then $\alpha^-(\ell_\alpha(t+))=\alpha^-(\ell_\alpha(t))$ and hence $\gamma^-$ is continuous at $t$. Hence (I$^-$),(II$^-$) and (III$^-$) for $\gamma^-_n$ reduces to showing $\gamma^-_n(t_n) \to \gamma(t)$.
(I$^-$) applied to $\alpha_n$ yields $\gamma^-_n(t_n) = \alpha^-_n(\tau^-_n) \to \{\alpha(\tau), \alpha^-(\tau)\}$.
Since $\ell_\alpha$ is right-constant at $t$, we have $\sigma^-(\tau)\leq t < \sigma(\tau)$, and we see that $\alpha(\tau) \neq \alpha^-(\tau)$. Suppose now that $\alpha^-_n(\tau^-_n)\to \alpha(\tau)$. It follows that $\sigma^-_n(\tau^-_n)\to \sigma(\tau)>t$, which contradicts $\sigma^-_n(\tau^-_n)\leq t_n \to t$.
Thus $\alpha^-_n(\tau^-_n)\to \alpha^-(\tau)$ meaning that $\gamma^-_n(t_n) \to \gamma^-(t) = \gamma(t)$.
We have proven the statement about $\Phi$ and thus our lemma.
\end{proof}

\section{CTRW Limit Theorems}
\label{sec:CTRW-limit}
\paragraph{Triangular arrays}
Following \cite{MeSc07Tal}, we define the triangular array
\begin{align}
\label{eq:triangular-array}
\Delta := \{(J_i^n,W_i^n); n,i\in\N\}
\end{align}
of $\Rd$-valued jumps $J_i^n$ and $\Rp$-valued waiting times $W_i^n$, which are random variables on some probability space $(\tilde \Omega, \tilde{\mathcal F}, \tilde{\mathbb P})$.
For every $n\in\N$, the sequence $\{(J^n_i,W^n_i)\}_{i\in\N}$ is assumed i.i.d.; note however that we do not assume that $J^n_i$ and $W^n_i$ are independent.
We define the sequence of row sum processes $(B^n,S^n) := \{(B^n_t,S^n_t)\}_{t\geq 0}$ via
\begin{align}
\label{eq:row-sum}
(B^n_t,S^n_t) := \sum_{k=1}^{\lfloor nt \rfloor} (J^n_k, W^n_k)
\end{align}
and note that all these processes start at $(0,0) \in \Rd\times\Rp$.
A standard argument involving finite collections of coordinate projections shows that $(B^n,S^n)$, seen as a map from $\tilde{\Omega}$ to $\mathbb D(\Rd\times\Rp)$, is measurable, i.e.\ an rcll path-valued random variable. We write $\pr_n$ for its distribution, which is a probability measure on $\mathbb D(\Rd\times\Rp)$, and note that $\pr_n(D_{\uparrow,u}) = 1$.
Similarly to \cite{MeSc07Tal},  we assume that the laws $\pr^n$ converge weakly on $(\mathbb D(\Rd\times\Rp),J)$ to a limit law $\pr$, which is the law of a \levy\ process. Setting
$(\Omega,\mathcal F^0) := \left(\mathbb D(\Rd\times\Rp),\mathcal D(\Rd\times\Rp)\right)$
and letting
\begin{align*}
(B_t,S_t) : \mathbb D(\Rd\times\Rp) &\to \Rd\times\Rp, \\
\omega &\mapsto \omega(t),
\end{align*}
for all $t\geq 0$, we produce the probability space $(\Omega,\mathcal F^0, \pr)$ on which $(B,S) = \{(B_t,S_t)\}_{t\geq 0}$ is a \levy\ process.
We let $\mathcal F^0_t$ be the $\sigma$-field generated by all $(B_u,S_u)$ where $0\leq u\leq t$, we denote the $\pr$-completion of $\mathcal F^0$ and $\mathcal F^0_t$ by $\mathcal F$ and $\mathcal F_t$, and work on the filtered probability space $(\Omega, \mathcal F, (\mathcal F_t)_{t\geq 0},\pr) $ from now on.
Observe that $S$ is a subordinator. We also make the following assumption for the rest of this work:
\begin{align} \label{eq:assumption}
   \textbf{The \levy\ process $(B,S)$ is not compound Poisson.}
\end{align}

\paragraph{Continuous Time Random Walks}

\begin{definition}
\label{def:XYAR}
For each $n\in\N$, the process $N^n:= \{N^n_t\}_{t\geq 0}$ given by
\begin{align*}
 N^n_t = \max\left\lbrace n\in\N\cup\{0\}: W_1 + \ldots + W_n \leq t\right\rbrace
\end{align*}
is called \textbf{renewal process}.
Moreover, the processes $X$, $Y$, $G$ and $D$ given by
\begin{align*}
\left(X^n_t,G^n_t\right) &= \sum_{k=1}^{N^n_t} (J^n_k, W^n_k),
&  (Y^n_t,D^n_t) &= \sum_{k=1}^{N^n_t+1} (J^n_k, W^n_k),
\end{align*}
are called \textbf{lagging CTRW}, \textbf{leading CTRW}, \textbf{last} and \textbf{next time of renewal}, respectively.
\end{definition}

\begin{lemma}
\label{lem:matching}
Let $\Phi$ and $\Psi$ be as in Definition \ref{def:Phi}. Then
\begin{align*}
\Phi\circ (B^n,S^n) &= (X^n,G^n), & \Psi\circ(B^n,S^n) &= (Y^n, D^n),  & n&\in \N.
\end{align*}
\begin{proof}
Let $L^n = (S^n)^{-1}$, and note that by definition of $N^n_t$, $S^n_t$ and the generalized inverse, the paths of $L^n$ are rcll step functions starting at $1/n$ with jumps of size $1/n$  occuring at the end of each renewal epoch; this yields $nL^n_t = N^n_t+1$, and the second equality follows.

Turning to the first statement, note that the left-hand limits of $(B^n,S^n)$ and $L^n$ satisfy
\begin{align*}
(B^n_{t-},S^n_{t-}) &= \sum_{k=1}^{\lceil nt-1 \rceil} (J^n_k,W^n_k),
& nL^n_{t-} &= N^n_{t-} + 1.
\end{align*}
From the formula $(\Phi(\beta,\sigma))^- = \alpha^-\circ \sigma^-$, valid for all $\alpha=(\beta,\sigma)\in D_{\uparrow,u}$, it follows that
$$(\Phi\circ(B^n,S^n))(t-) = \sum_{k=1}^{N^n_{t-}} (J^n_k,W^n_k) = \left(X^n_{t-},G^n_{t-}\right).$$
Taking right-hand limits again yields the first equality.
\end{proof}
\end{lemma}

We have prepared everything for the proof of the first main result:
\begin{theorem}
\label{th:CTRW-limit}
Define $(X,G) = \Phi(B,S)$ and $(Y,D) = \Psi(B,S)$. Then
\begin{align*}
  (X^n,G^n) &\Rightarrow (X,G), &(Y^n,D^n)&\Rightarrow (Y,D),
\end{align*}
where ``$\Rightarrow$'' denotes weak convergence in $(\mathbb D(\Rd\times\Rp),J)$ as $n\to\infty$.
\begin{proof}
By lemma \ref{lem:measurable}, $D_{\uparrow,u}$ is measurable, and we find that $\pr(D_{\uparrow,u})= \pr^n(D_{\uparrow,u}) = 1$ for all $n\in\N$.
Hence by \cite[ch.3, cor.3.2]{EtK86} the restrictions of $\pr^n$ to $D_{\uparrow,u}$ (endowed with the relative topology and Borel $\sigma$-field), converge weakly to the same restriction of $\pr$.
By proposition \ref{prop:PhiCont}, $D_{\upuparrows,u}$ is contained in the set of continuity points of $\Phi$ and $\Psi$,
and $\pr(D_{\upuparrows,u})=1$ by assumption (\ref{eq:assumption}).
Then Lemma~\ref{lem:matching} and the continuous mapping theorem (see e.g.\ \cite[p.30]{Bil99}) yield the statement.
\end{proof}
\end{theorem}

\begin{remark}
Silvestrov and Teugels have shown the weak $J$-convergence of $(Y^n,D^n)$ to $(Y,D)$ for $d=1$ in \citep[th~3.4]{SiTe04Ltm}, using a compactness approach.
\end{remark}


\begin{remark}
In our notation, Becker-Kern et.\ al have shown \citep[th~3.1]{BeMS04Ltca} the weak convergence
\begin{align*}
  X^n \Rightarrow Y ~~~~\text{ in } \mathbb D(\Rd)
\end{align*}
with respect to the weaker $M_1$-topology under the assumption
\begin{align*}
 \text{disc}(B) \cap \text{disc}(S) = \emptyset ~~~~ \text{a.s.},
\end{align*}
where $\text{disc}(B)$ and $\text{disc}(S)$ denote the (random) sets of discontinuities of $B$ and $S$, respectively. This assumption is typically only satisfied by \emph{un}coupled CTRW limits.
The following observation explains why theorem \ref{th:CTRW-limit} does not contradict their limit theorem:
\end{remark}

\begin{lemma} \label{lem:X=Y}
Let $(\beta,\sigma) \in D_{\upuparrows,u}$ and write $\Phi(\beta,\sigma) = (\xi,\iota)$, $\Psi(\beta,\sigma) = (\eta,\upsilon)$.
If $\text{disc}(\beta) \cap \text{disc}(\sigma) = \emptyset$, then $\xi = \eta$.
\begin{proof}
Let $\ell = \sigma^{-1}$, and see that $\ell$ is continuous since $\sigma$ is increasing.
Then by definition of $\Phi$ and $\Psi$ we have $\xi = (\beta^-\circ \ell)^+$ and $\eta = \beta\circ\ell $.
Both paths start at the same point $\beta(0)\in\Rd$, hence it suffices to show that their left limits coincide at every $t>0$.
We find
\begin{align*}
  \eta(t-) =
\begin{cases}
\beta^-(\ell(t)) & \text{ if $\ell$ is left-increasing at $t$}
\\ \beta(\ell(t)) &\text{ else }
\end{cases}
\end{align*}
where ``left-increasing at $t$'' means $s<t \Rightarrow \ell(s) <\ell(t)$.
In the second case, $\ell(t) \in \text{disc}(\sigma)$, hence $\ell(t) \notin \text{disc}(\beta)$ by assumption, so that $\beta(\ell(t)) = \beta^-(\ell(t))$. Thus we have shown
$\eta^- = \beta^-\circ\ell = ((\beta^-\circ \ell)^+)^- = \xi^-$.
\end{proof}
\end{lemma}

\begin{remark}
Under the assumption that $B$ and $S$ are independent, the process $B\circ L^-$ appears as a CTRW limit in \cite{BeMS04Ltc}. The interesting generalization there is that $S$ is \emph{not} assumed strictly increasing. Convergence has only been shown for all finite dimensional distributions, and it is an interesting open question whether or not the rcll versions convergence weakly on $\mathbb D(\Rd)$. The continuous mapping theorem does not give a simple answer in this case, as the mapping $\Phi$ is not continuous on $D_{u,\uparrow}$, even if the uniform convergence topology is chosen on the domain and the $M_1$-topology is chosen on the codomain. To see this, let $(\beta_n, \sigma_n)$ and $(\beta,\sigma)$ be given by
\begin{align*}
\beta_n(t) &= \beta(t), & \beta(t) &= \begin{cases}
             1, & 1 \leq t < 2
\\ 0, & \text{else}
            \end{cases}
,
\\
\sigma_n(t) &=
\begin{cases}
 2-1/n, & 0\leq t <1
\\ 2+1/n, & 1\leq t < 2 + 1/n
\\ t, & 2+1/n\leq t
\end{cases}
, &
\sigma(t) &=
\begin{cases}
2, & 0\leq t < 2
\\ t, & 2\leq t
\end{cases}.
\end{align*}
Then $(\beta_n,\sigma_n) \to (\beta,\sigma)$ with the (strongest) uniform topology, and $\beta\circ\sigma^{-1}\equiv 0$, but
$\beta_n\circ \sigma_n^{-1}(2) = 1$ for all $n$.
\end{remark}

\section{CTRW Limit Laws}

\begin{definition}
The $\Rp$-valued processes $A$ and $R$ given by
\begin{align*}
  A_t &= t - G_t, & R_t = D_t - t
\end{align*}
are called \textbf{age} and \textbf{remaining lifetime}, respectively.
The random set
\begin{align*}
 \mathbf M = \{(t,\omega)\in\Rp\times\Omega: t = S_u(\omega) \text{ for some } u\geq 0\}.
\end{align*}
is called \textbf{regenerative set}.
Its $\omega$-slices and $t$-slices are the sets
\begin{align*}
\mathbf M^\omega &= \{t\in\Rp: (t,\omega) \in \mathbf M\},
& \mathbf M_t &= \{\omega\in\Omega: (t,\omega) \in \mathbf M\}.
\end{align*}
\end{definition}

Note that for $\pr$-almost every $\omega$,
\begin{align*}
  G_t(\omega) &= \sup\left([0,t]\cap\mathbf M^\omega\right), & D_t(\omega) &= \inf([t,\infty)\cap \mathbf M^\omega).
\end{align*}
The term ``regenerative'' refers to the property that the part of $\mathbf M$ to the right of any $D_t(\omega) \in \mathbf M^\omega$ has the same distribution as $\mathbf M$. This can be easily inferred from the strong Markov property of $S$; for details see e.g.\ \cite{Bert97SEa}.

The set  $\mathbf M^\omega$ can be interpreted as the range of the path $S_\cdot(\omega)$.
Almost every such path is rcll and has countably many jumps, hence the complement $(\mathbf M^\omega)^\complement$ can be written as a countable union of intervals of the form $[g_i,d_i)$; such intervals are commonly termed \textbf{contiguous to $\mathbf M$}.
The collection of all such $g_i$ (resp.\ $d_i$) defines the set $\mathbf G^\omega$ (resp.\ $\mathbf D^\omega$), which in turn defines a random set $\mathbf G$ (resp.\ $\mathbf D$).
In what follows, if we apply a topological operation (e.g.\ ``complement'', ``union'' or ``closure'') to a random set $\subset \Rp\times\Omega$, then we mean the application of this operation to all $\omega$-slices $\subset \Rp$; for instance $\overline {\mathbf M} = \mathbf M \cup \mathbf G$.
\begin{lemma}
Fix $t\geq 0$.
\label{lem:fixed-disc}
\begin{enumerate}[i.]
\item
\label{it:lem4}
On $\mathbf M$ we have $Y = X$ and $R = A \equiv 0$.
\item
\label{it:lem5}
$\pr(\mathbf G_t) = \pr(\mathbf D_t) = 0$.
\item
\label{it:lem1}
$\pr(\{\omega:L_t(\omega) \in \text{disc}(B,S)^\omega \setminus \text{disc}(S)^\omega\}) = 0$.
 \item
\label{it:lem3}
$ \pr\left((X,A)^-_t = (X,A)_t\right) = 1$, i.e.\ $(X,A)$ admits no fixed discontinuities.
\item
\label{it:lem2}
$(Y_t,D_t) - (X_t,G_t) = \Delta(B,S)_{L_t}$ almost surely.
\end{enumerate}
\begin{proof}
\ref{it:lem4}.\
For almost all $\omega\in\Omega$, every point in $\mathbf M^\omega$ is a right-limit point. Hence $L_{\cdot}(\omega)$ is right-increasing at $t\in \mathbf M^\omega$, i.e.\ $t<u \Rightarrow L_t<L_u$, and thus
$(X_t,G_t) = (B,S)^-(L_{t+}) = (B,S)(L_t) = (Y_t,D_t)$. As, by definition, $G_t \leq t \leq D_t$, we have $G_t=t=D_t$, i.e.\ $A_t = R_t = 0$.

\ref{it:lem5}.\ $\pr(\mathbf G_t)=0$ has been shown in \cite[proof of prop.1.9~(ii)]{Bert97SEa}.
Using a similar argument based on the compensation formula, we find that
\begin{align*}
\pr(\mathbf D_t) = \pr(\exists u\geq 0: S^-_u<t, \Delta S_u = t-S^-_u) = \int_{[0,t)}  \Pi(\{t-y\})U(dy),
\end{align*}
where $U$ and $\Pi$ denote the $0$-potential (or ``renewal measure'') and the \levy-measure of the subordinator $S$, respectively. Since the integrand equals zero except at at most countably many points and since $U$ has no atoms (\cite[p.10]{Bert97SEa}), the integral vanishes.

\ref{it:lem1}.\ We abbreviate
$\mathbf K := {\rm disc}(B,S)\setminus{\rm disc}(S)$. For each $u\geq 0$, $\mathbf K_u$ is measurable with respect to the $\sigma$-field
$\sigma(Z_v:0\leq v\leq u)$ generated by the $\Rd$-valued process $Z$:
\begin{align*}
Z_t(\omega) := \sum_{0\leq u \leq t} \mathbf 1\left\lbrace\Delta(B,S)_u(\omega)\in\Rd\times\{0\}\right\rbrace \Delta B_u(\omega).
\end{align*}
By the Poisson process nature of the jumps of $(B,S)$, the
processes $Z$ and $(B,S)-(Z,0)$ are independent.
In particular, $\mathbf K$ is independent of $S$ and hence of $L$. This yields
\begin{align*}
&\pr(\{\omega: L_t(\omega) \in \mathbf K^\omega\}) = \E[\pr(\{\omega: L_t(\omega) \in \mathbf K^\omega\}|L_t)]
\\ &= \int_0^\infty \pr(L_t\in du) \pr(\{\omega: u \in \mathbf K^\omega\}|L_t = u)
= \int_0^\infty \pr(L_t\in du)\pr(\mathbf K_u).
\end{align*}
Since the \levy\ process $Z$ has no fixed discontinuities, $\pr(\mathbf K_u) = 0$ for all $u\geq 0$, and the integral vanishes.

\ref{it:lem3}.\
Observe that if $\omega\in\Omega$ is fixed then $L$ is constant on every open interval which is contiguous to  $\overline{\mathbf M}$. Hence the same is true for $(X,G)$, and thus $\text{disc}(X,A) \cap \overline{\mathbf M}^\complement = \emptyset$. Thus we established
\begin{align*}
 \text{disc}(X,A) \subset \overline {\mathbf M} = \mathbf G \cup (\mathbf M\setminus \mathbf D) \cup \mathbf D,
\end{align*}
and according to part \ref{it:lem5}, it now suffices to show
$\pr((\text{disc}(X,A) \cap (\mathbf M\setminus \mathbf D))_t)= 0$. Suppose $(t,\omega) \in \mathbf M\setminus \mathbf D$. Then $t$ is both left- and right-limit point, and one consequence is that $G$ and thus $A$ are continuous at $t$.  Another consequence is that $L_{\cdot}(\omega)$ is both left- and right-increasing at $t$, i.e.\ $s<t<u \Rightarrow L_s<L_t<L_u$. Since moreover $L$ is continuous, and since $G^- = S^-\circ L$ is continuous at $t$, $S$ must be continuous at $L_t$. Item \ref{it:lem1} then implies that $B$ is continuous at $L_t$ a.s., and hence $X^-=B^-\circ L$ and $X$ are both continuous at $t\in\mathbf M\setminus \mathbf D$ a.s..

\ref{it:lem2}.\
By definition, $(Y_t,D_t) - (X^-_t,G^-_t) = \Delta(B,S)_{L_t}$, so the statement follows from item~\ref{it:lem3} and $A_t = t-G_t$.
\end{proof}
\end{lemma}

\paragraph{\levy-characteristics}

For a bounded Borel measure $\nu$ on $\Rd\times\Rp$ we define its \textbf{Fourier-Laplace transform} (FLT) as the map
\begin{align*}
\Rd\times\Rp\ni(k,s)\mapsto \int_{\Rd\times\Rp} \exp(i \langle x,k \rangle - st) \nu(dx,dt) \in \mathbb C,
\end{align*}
with $\langle\cdot,\cdot\rangle$ denoting the inner product in $\Rd$.
Lemma 2.1 in \cite{BeMS04Ltca}, which is derived from a general theorem on integral transforms on semigroups in \cite{Ress91Sip}, states that for $t\geq 0$ the FLT of the law of $(B_t,S_t)$  is $\exp(-t\psi(k,s))$, where
\begin{align*}
 \psi(k,s) = i\langle b,k \rangle + \gamma s + \sigma^2(k) +
 \int \left(1 - \exp\left(i\langle x,k \rangle - st \right) + \dfrac{i\langle k,x \rangle}{1+\|x\|^2} \right) \Pi(dx,dt)
\end{align*}
for some $(b,\gamma) \in \Rd\times\Rp$, a quadratic form $\sigma^2$ on $\Rd$ and a Borel measure $\Pi$ on $\Rd\times\Rp \setminus\{(0,0)\}$ which assigns finite mass to sets bounded away from $(0,0)$ and satisfies
\begin{align*}
 \int_{0<t+\|x\|^2<1} (\|x\|^2 +t) \Pi(dx,dt) < \infty.
\end{align*}
The distribution of $(B_1,S_1)$ uniquely determines $b$, $\gamma$, $\sigma^2$ and $\Pi$, which are called \textbf{drift of $B$}, \textbf{drift of $S$}, \textbf{diffusivity of $B$} and \textbf{jump measure} of $(B,S)$, respectively.
We write $\overline \Pi(a) := \Pi(\Rd\times(a,\infty))$, $a\geq 0$, for the \textbf{tail function} of $\Pi$.
Note that assumption (\ref{eq:assumption}) is then equivalent to
``$ \overline \Pi(0) = \infty$ or $\gamma > 0$''. The \textbf{potential measure} $U$ of $(B,S)$ is the unique measure on $\Rd\times\Rp$ which satisfies
\begin{align*}
 \int f(b,s) U(db,ds) = \E\left[\int_0^\infty f((B,S)_t)dt\right],
~~~ f\in p\mathcal B(\Rd\times\Rp)
\end{align*}
where $p\mathcal B(\Rd\times\Rp)$ is the set of all real-valued $\mathcal B(\Rd\times\Rp)$-measurable functions which only attain non-negative values.

For fixed $t\geq0$, we want to find the joint law of $(X_t,A_t,Y_t,R_t)$.
As we have seen that $(X_t,A_t)$ does not admit any fixed discontinuities, this law follows easily from the joint law of $(X^-_t,A^-_t,Y_t,R_t) = (B^-_{L_t},t-S^-_{L_t},B_{L_t},S_{L_t}-t)$.
Observing that the \levy\ process $(B,S)$ can be viewed as a special case of a Markov additive process \cite{cC72Map}, we can find this law in \cite{cC76EdM} (or in \cite{Mais77Cdt}, where an easier proof is given); however, only for $t$ outside of some Lebesgue nullset. With a few additional assumptions on the regularity of $U$, we show that the formula in the latter reference holds true for \emph{all} $t\geq0$.

\begin{proposition}
\label{lem:R=0}
Suppose that $\gamma > 0$, that the laws of $(B_t,S_t)$ are Lebesgue-absolutely continuous for $t>0$, and that the Lebesgue density $u(\cdot,\cdot)$ of $U$ is continuous. Then
\begin{align*}
\pr(Y_t\in C, R_t = 0) = \gamma \int_C u(b,t)db, ~~~~ C\in \mathcal B(\Rd).
\end{align*}
\begin{proof}
Note that absolute continuity of the laws $(B_t,S_t)$ ($t>0$) implies absolute continuity of $U$ (\cite[41.13]{Sat99}).
The marginal $U(\Rd, \cdot)$ of $U$ is the potential of the subordinator $S$ and is usually called ``renewal measure''. Its Laplace-transform satisfies
\begin{align*}
\int_0^\infty e^{-s x}U(\Rd,dx) = \frac{1}{\psi(0,s)}
\sim \frac{1}{\gamma s}, ~~~ (s\to \infty).
\end{align*}
A Tauberian theorem then implies that
\begin{align}
\label{eq:tauber}
 U(\Rd\times[0,\eps]) \sim \eps/\gamma, ~~~ (\eps \to 0+).
\end{align}
For any $\delta >0$ define $B_\delta = \{x\in \Rd: \Vert x \Vert < \delta\}$. 
Then using $L_\eps \leq \eps / \gamma$ and the stochastic continuity of $B$ at $0$ we have
\begin{align*}
U(B_\delta^\complement \times[0,\eps])
= \E \left[\int_0^{L_\eps}\1\{B_t \geq \delta \}\right]
\leq \int_0^{\eps/\gamma} \pr (B_t \geq \delta) dt = o(\eps).
\end{align*}
Without loss of generality, let $C\subset \Rd$ be open; then (\ref{eq:tauber}) can be refined to 
\begin{align} \label{eq:tauber2}
 \lim_{\eps \downarrow 0}\frac{U(C\times[0,\eps])}{\eps} 
= \frac{\1\{0\in C\}}{\gamma}
\end{align}
if $0\notin \del C$. 
The continuity of $u$ and an application of the Markov property at the time $L_t$ yield
\begin{align*}
\int_{C} u(b,t) db &= \lim_{\eps\downarrow 0} \eps^{-1}U(C\times(t,t+\eps])
= \lim_{\eps\downarrow 0} \eps^{-1}\E\left[\int_{L_t}^\infty\1\{(B,S)_u\in C\times (t,t+\eps]\}du\right]
\\ 
&=\lim_{\eps\downarrow 0}\eps^{-1}\int_{\Rd\times[t,t+\eps]}\pr((B,S)_{L_t}\in (db,ds)) 
\left(U((C-b)\times[0,t+\eps-s)\right).
\end{align*}
The last displayed integral, if instead taken over the set $\Rd\times(t,t+\eps]$, is bounded by
\begin{align*}
\pr(S_{L_t}\in (t,t+\eps]) U(\Rd\times[0,\eps)), 
\end{align*}
which is of order $o(\eps)$ by (\ref{eq:tauber}).
Hence, by dominated convergence and (\ref{eq:tauber2}),
\begin{align*}
\int_C u(b,t)db = \int_{\Rd\times\{t\}}\pr((B,S)_{L_t}\in (db,ds))
\gamma^{-1}\1\{b\in C\} = \gamma^{-1}\pr(Y_t \in C, D_t = t),
\end{align*}
where we have assumed without loss of generality that the boundary of $C$ has Lebesgue measure $0$.
\end{proof}
\end{proposition}

\begin{proposition}
\label{prop:R>0}
Let $f\in B_+\left(\Rd\times\Rp\times\Rd\times\Rp\right)$. Then
\begin{align}\label{eq:prop:R>0}
\begin{split}
&\E[f(X_t,A_t,Y_t,R_t)\mathbf 1\{R_t > 0\}] =
\\&= \int_{\Rd\times[0,t]}U(dy,ds)
\int_{\Rd\times[t-s,\infty)}\Pi(d\xi,d\eta) f(y,t-s,y+\xi,s+\eta-t)
\end{split}
\end{align}
\begin{proof}
An application of the compensation formula \cite[prop.~XII.1.10]{ReY99} yields
\begin{align}
\label{eq:comp-formula}
 \E\left[\sum_{s>0}F\left((B,S)^-_s, \Delta(B,S)_s\right)\right]
= \E\left[\int\limits_0^\infty ds \int F\left((B,S)^-_s,(\xi,\eta)\right)\Pi(d\xi,d\eta)\right],
\end{align}
valid for all $F\in p\mathcal B(\Rd\times\Rp\times\Rd\times\Rp)$ such that $F(\cdot, \cdot, 0,0) = 0$.
Letting
\begin{align*}
 F\left((x,z),(\xi,\eta)\right) = f(x,t-z,x+\xi,z+\eta-t)\mathbf 1\{z\leq t, z+\eta > t\},
\end{align*}
we find that $s=L_t$ yields the only possibly non-zero summand.
After an application of lemma \ref{lem:fixed-disc}(\ref{it:lem3}) we can match the left sides of (\ref{eq:prop:R>0}) and (\ref{eq:comp-formula}).
The \levy\ process $(B,S)$ has no fixed discontinuities, and a short calculation shows that the right sides are equal as well.
\end{proof}
\end{proposition}

We are now ready to give our second main result:

\begin{theorem}
\label{th:munu}
For $t\geq 0$ and $f\in p\mathcal B(\Rd\times\Rp\times\Rd\times\Rp)$, we have
\begin{align} \label{eq:XAYR}
\begin{split}
\E[f(X_t,A_t,Y_t,R_t)] &= 
\gamma \int f(y,0,y,0) u(dy,t) 
\\&+  \int\limits_{\Rd\times[0,t]}U(dy,ds)
      \int\limits_{\Rd\times[t-s,\infty)} \Pi(d\xi,d\eta)f(y,t-s,y+\xi,s+\eta-t). 
\end{split}
\end{align}
\begin{proof}
Observe that $\{(t,\omega): R_t(\omega) = 0\} = \mathbf M$. Combining lemma~\ref{lem:fixed-disc}(\ref{it:lem4}), lemma~\ref{lem:R=0} and proposition~\ref{prop:R>0} then yields the above formulae.
\end{proof}
\end{theorem}

Note that the joint law of only $(X_t,A_t)$ admits the slightly simpler formula
\begin{align*}
\E[f(X_t,A_t)] &= \gamma \int f(y,0) u(dy,t) +
 \int\limits_{\Rd\times[0,t]}U(dy,ds)f(y,t-s)\bar\Pi(t-s).
\end{align*}
Moreover, note that in the uncoupled case the jump measure $\Pi$ is supported by $\Rd\times\{0\} \cup \{0\} \times \Rp$, whence the integration variable $\xi$ in (\ref{eq:XAYR}) vanishes and $X_t$ and $Y_t$ have the same law, which agrees with lemma~\ref{lem:X=Y}.

\begin{remark}
The law of $X_t$ has appeared in \cite[th.4.1]{BeMS04Ltca} and \cite[th.3.6]{MeSc07Tal}; however, it has been overlooked that in general 
$X_t$ is not equal to $B\circ L(t) = Y_t$. 
Moreover, we do not need to impose any growth condition on $\overline \Pi(t)$, and  we have relaxed assumptions from ``$\overline \Pi(0) = \infty$ and $\gamma = 0$'' to ``$\overline \Pi(0) = \infty$ or $\gamma > 0$.'' We believe that theorem~\ref{th:munu} holds true also in the case ``$\overline \Pi(0) < \infty$ and $\gamma = 0$.'' But then $X$ and $Y$ are essentially CTRWs and not CTRW limit processes, and we do not investigate this any further.
\end{remark}

Finally, we remark that CTRWs with infinite mean waiting times are models for physical processes which exhibit very slow relaxation and ``ageing''; see \cite{Bouc92Web,BaCh03Act}, and also \cite{vCtpo} for a mathematical account. We have yet to connect the age process with the ageing phenomenon in our future work.

\bibliographystyle{elsarticle-harv}
\bibliography{diss.bib}

\end{document}